\documentclass{amsart}
\usepackage{graphicx}

\newtheorem{thm}{Theorem}[section]
\newtheorem*{thm*}{Theorem}

\newtheorem{lem}[thm]{Lemma}

\theoremstyle{definition}

\theoremstyle{remark}
\newtheorem{rem}[thm]{Remark}

\numberwithin{equation}{section}

\allowdisplaybreaks

\begin{document}

\title[An example of spectral phase transition phenomenon]{An example of spectral phase transition phenomenon
in a class of Jacobi matrices with periodically modulated weights}

\author{Sergey Simonov}

\address{Department of Mathematical Physics, Institute of Physics, St. Petersburg University,
Ulianovskaia 1, 198904, St. Petergoff, St. Petersburg, Russia
} \email{sergey\_simonov@mail.ru}


\date{}

\begin{abstract}
We consider self-adjoint unbounded Jacobi matrices with diagonal
$q_n=n$ and weights $\lambda_n=c_n n$, where $c_n$ is a
2-periodical sequence of real numbers. The parameter space is
decomposed into several separate regions, where the spectrum is
either purely absolutely continuous or discrete. This constitutes
an example of the spectral phase transition of the first order. We
study the lines where the spectral phase transition occurs,
obtaining the following main result: either the interval
$(-\infty;\frac12)$ or the interval $(\frac12;+\infty)$ is covered
by the absolutely continuous spectrum, the remainder of the
spectrum being pure point. The proof is based on finding
asymptotics of generalized eigenvectors via the Birkhoff-Adams
Theorem. We also consider the degenerate case, which constitutes
yet another example of the spectral phase transition.
\end{abstract}

\maketitle

\section{Introduction}

In the present paper we study a class of Jacobi matrices with
unbounded entries: a linearly growing diagonal and periodically
modulated linearly growing weights.

We first define the operator $J$ on the linear set of vectors
$l_{fin}(\mathbb{N})$ having finite number of non-zero elements:
\begin{equation}
\label{maineqn}
(Ju)_n=\lambda_{n-1}u_{n-1}+q_n u_n+\lambda_{n}u_{n+1},\;n\geq2
\end{equation}
with the initial condition $(Ju)_1=q_1u_1+\lambda_1u_2$, where
$q_n=n,\;\lambda_n=c_nn$, and $c_n$ is a real 2-periodic sequence,
generated by the parameters $c_1$ and $c_2$.

Let $\{e_n\}_{n\in\mathbb{N}}$ be the canonical basis in
$l^2(\mathbb{N})$. With respect to this basis the operator $J$
admits the following matrix representation:
$$
J = \begin{pmatrix}
q_1       & \lambda_1 & 0         & \cdots \\
\lambda_1 & q_2       & \lambda_2 & \cdots \\
0         & \lambda_2 & q_3       & \cdots \\
\vdots    & \vdots    & \vdots    & \ddots
\end{pmatrix}
$$
Due to the Carleman condition \cite{Bz}
$\sum^{\infty}_{n=1}\frac1{\lambda_n}=\infty$, the operator $J$ is
essentially self-adjoint. We will therefore assume throughout the
paper, that $J$ is a closed self-adjoint operator in
$l^2(\mathbb{N})$, defined on its natural domain $D(J)=\{u\in
l^2(\mathbb{N}):Ju\in l^2(\mathbb{N})\}$.

We base our spectrum investigation on the subordinacy theory due
to Gilbert and Pearson \cite{GP}, generalized to the case of
Jacobi matrices by Khan and Pearson \cite{KP}. Using this theory,
we study an example of spectral phase transition of the first
order. This example was first obtained by Naboko and Janas in
\cite{JN1} and \cite{JN2}. In cited articles, the authors managed
to demonstrate that the space of parameters
$(c_1;c_2)\in\mathbb{R}^2$ can be naturally decomposed into a set
of regions of two types. In the regions of the first type, the
spectrum of the operator $J$ is purely absolutely continuous and
covers the real line $\mathbb{R}$, whereas in the regions of the
second type the spectrum is discrete.

Due to \cite{JN1} and \cite{JN2}, spectral properties of
Jacobi matrices of our class are determined by the location of the
point zero relative to the absolutely continuous spectrum of a
certain periodic matrix $J_{per}$, constructed based on the
modulation parameters $c_1$ and $c_2$. In our case this leads to:
\begin{equation*}
J_{per} = \begin{pmatrix}
1   & c_1 & 0   & 0   & \cdots \\
c_1 & 1   & c_2 & 0   & \cdots \\
0   & c_2 & 1   & c_1 & \cdots \\
0   & 0   & c_1 & 1   & \cdots \\
\vdots & \vdots & \vdots & \vdots & \ddots
\end{pmatrix}
\end{equation*}
Considering the characteristic polynomial
\begin{equation*}
d_{J_{per}}(\lambda)=Tr\left(
\begin{pmatrix}0 & 1 \\ -\frac{c_1}{c_2} & \frac{\lambda-1}{c_2} \end{pmatrix}
\begin{pmatrix}0 & 1 \\ -\frac{c_2}{c_1} & \frac{\lambda-1}{c_1} \end{pmatrix}
\right)=\frac{(\lambda-1)^2-c_1^2-c_2^2}{c_1c_2},
\end{equation*}
the location of the absolutely continuous spectrum
$\sigma_{ac}(J_{per})$ of $J_{per}$ can then be determined from
the following condition \cite{Bz}:
\begin{equation}                                      \label{crit}
\lambda\in\sigma_{ac}(J_{per}) \Leftrightarrow
\left|d_{J_{per}}(\lambda)\right|\leq2.
\end{equation}
This leads to the following result \cite{JN2}, concerning the
spectral structure of the operator $J$.

\vspace{7 pt} \it
If $\left|d_{J_{per}}(0)\right|<2$, then the
spectrum of the operator $J$ is purely absolutely continuous,
covering the whole real line.

If, on the other hand, $\left|d_{J_{per}}(0)\right|>2$, then the
spectrum of the operator $J$ is discrete.
\rm \vspace{7 pt}

Thus, the condition $\left|\frac{1-c_1^2-c_2^2}{c_1c_2}\right|=2$,
equivalent to $\{ \;|c_1|+|c_2|=1 \text{ or } ||c_1|-|c_2||=1 \;
\}$, determines the boundaries of the above mentioned regions on
the plane $(c_1;c_2)$ where one of the cases holds and
the spectrum of the operator $J$ is either purely
absolutely continuous or discrete (see fig. 1 on page
\pageref{fig}).

Note also, that Jacobi matrices with modulation parameters equal
to $\pm c_1$ and $\pm c_2$ are unitarily equivalent. Thus the
situation can be reduced to studying the case $c_1,\ c_2>0$.

In the present paper we attempt to study the spectral structure on
the lines, where the spectral phase transition occurs, i.e., on
the lines separating the aforementioned regions.

The paper is organized as follows.

Section 2 deals with the calculation of the asymptotics of
generalized eigenvectors of the operator $J$. This calculation is
mainly based on the Birkhoff-Adams Theorem \cite{El}. The
asymptotics are then used to characterize the spectral structure
of the operator via the Khan-Pearson Theorem \cite{KP}. It turns
out, that on the lines where the spectral phase transition occurs
the spectrum is neither purely absolutely continuous nor pure
point, but a combination of both.

In Section 3, we attempt to ascertain whether the pure point part
of the spectrum is actually discrete. In doing so, we establish a
criterion that guarantees that the operator $J$ is semibounded
from below, for all $(c_1;c_2)\in\mathbb{R}^2$. This
semiboundedness is then used in conjunction with classical methods
of operator theory to prove, that in at least one situation the
discreteness of the pure point spectrum is guaranteed.

Section 4 is dedicated to the study of the degenerate case, i.e.,
the case when one of the modulation parameters turns to zero. In
this situation, one can explicitly calculate all eigenvalues
of the operator. On this route we obtain yet another "hidden"
example of the spectral phase transition of the first order as the
point $(c_1;c_2)$ moves along one of the critical lines in the
space of parameters.

\section{Generalized eigenvectors and the spectrum of the operator J}

In this Section, we calculate asymptotics of generalized
eigenvectors of the operator $J$. Consider the recurrence relation
\cite{KP}
\begin{equation}                                                            \label{u eqn}
\lambda_{n-1}u_{n-1}+(q_n-\lambda)u_n+\lambda_{n}u_{n+1}=0,\;n\geq2.
\end{equation}
We reduce it to a form such that the Birkhoff-Adams Theorem is
applicable. To this end, we need to have a recurrence relation of
the form:
\begin{equation}                                                            \label{birkeqn}
x_{n+2}+F_1(n)x_{n+1}+F_2(n)x_n=0,\;n\geq1,
\end{equation}
where $F_1(n)$ and $F_2(n)$ admit the following asymptotical
expansions as $n\rightarrow\infty$:
\begin{equation}                                                            \label{birk as}
F_1(n)\sim \sum_{k=0}^\infty \frac{a_k}{n^k},\; F_2(n)\sim
\sum_{k=0}^\infty \frac{b_k}{n^k}
\end{equation}
with $b_0\neq 0$. Consider the characteristic equation
$\alpha^2+a_0\alpha+b_0=0$ and denote its roots $\alpha_1$ and
$\alpha_2$. Then \cite{El}:

\begin{thm*}[Birkhoff-Adams]
There exist two linearly independent solutions $x_n^{(1)}$ and
$x_n^{(2)}$ of the recurrence relation \eqref{birkeqn} with the
following asymptotics as $n\rightarrow \infty$:

1.
$$
x_n^{(i)}=\alpha_i^n n^{\beta_i}\left(1+O \left( \frac1n \right)
\right), \ i=1,2,
$$
if the roots $\alpha_1$ and $\alpha_2$ are different, where
$\beta_i=\frac{a_1 \alpha_i+b_1}{a_0 \alpha_i+2 b_0}$ ,\; i=1,2.

2.
$$
x_n^{(i)}=\alpha^n e^{\delta_i \sqrt{n}}n^{\beta} \left(1+O
\left(\frac1{\sqrt{n}} \right) \right), \ i=1,2,
$$
if the roots $\alpha_1$ and $\alpha_2$ coincide,
$\alpha:=\alpha_1=\alpha_2$, and an additional condition
$a_1\alpha+b_1\neq0$ holds, where
$\beta=\frac14+\frac{b_1}{2b_0},\;\delta_1=2\sqrt{\frac{a_0a_1-2b_1}{2b_0}}=-\delta_2$.
\end{thm*}

This Theorem is obviously not directly applicable in our case, due
to wrong asymptotics of coefficients at infinity. In order to deal
with this problem, we study a pair of recurrence relations,
equivalent to \eqref{u eqn}, separating odd and even components of
a vector $u$. This allows us to apply the Birkhoff-Adams Theorem
to each of the recurrence relations of the pair, which yields the
corresponding asymptotics. Combining the two asymptotics together,
we then obtain the desired result for the solution of \eqref{u
eqn}.

Denoting $ v_k:=u_{2k-1}$ and $ w_k:=u_{2k}$, we rewrite the
recurrence relation \eqref{u eqn} for the consecutive values of
$n$: $n=2k$ and $n=2k+1$.
\begin{equation*}
\lambda_{2k-1}v_k+(q_{2k}-\lambda)w_k+\lambda_{2k}v_{k+1}=0,
\end{equation*}
\begin{equation*}
\lambda_{2k}w_k+(q_{2k+1}-\lambda)v_{k+1}+\lambda_{2k+1}w_{k+1}=0.
\end{equation*}

Then we exclude $w$ in order to obtain the recurrence relation for
$v$:
$$
w_k=-\frac{\lambda_{2k-1}v_k+\lambda_{2k}v_{k+1}}{q_{2k}-\lambda},
$$
\begin{equation}                                    \label{v eqn}
v_{k+2}+P_1(k)v_{k+1}+P_2(k)v_k=0,\; k\geq1,
\end{equation}
where
$$
P_1(k)=
\frac{q_{2k+2}-\lambda}{q_{2k}-\lambda}\frac{\lambda_{2k}^2}{\lambda_{2k+1}\lambda_{2k+2}}-
\frac{(q_{2k+1}-\lambda)(q_{2k+2}-\lambda)}{\lambda_{2k+1}\lambda_{2k+2}}
+ \frac{\lambda_{2k+1}}{\lambda_{2k+2}},
$$
$$
P_2(k)= \frac{q_{2k+2}-\lambda}{q_{2k}-\lambda}
\frac{\lambda_{2k-1}\lambda_{2k}}{\lambda_{2k+1}\lambda_{2k+2}},
$$
In our case ($\lambda_n=c_n n$ and $q_n=n$) this yields the
following asymptotic expansions (cf. \eqref{birk as}) for
$P_1(k)$ and $P_2(k)$ as $k$ tends to infinity:
$$
P_1(k) = \sum_{j=0}^{\infty}{\frac{a_j}{k^j}}, \;
P_2(k) = \sum_{j=0}^{\infty}{\frac{b_j}{k^j}}
$$
with
\begin{equation}                        \label{cfs}
a_0 = \frac{c_1^2+c_2^2-1}{c_1c_2}, \; a_1 =
-\frac{c_1^2+c_2^2-2\lambda}{2c_1c_2}=-\frac{a_0}2+\frac{\lambda-\frac12}{c_1c_2},
\end{equation}
$$
b_0 = 1, \; b_1 = -1.
$$
The remaining coefficients $\{a_j\}_{j=2}^{+\infty},
\ \{b_j\}_{j=2}^{+\infty}$ can also be calculated explicitly.

On the same route one can obtain the recurrence relation for even
components $w_k$ of the vector $u$:
\begin{equation}                                        \label{w eqn}
w_{k+2}+R_1(k)w_{k+1}+R_2(k)w_k=0,\; k\geq1.
\end{equation}
Note, that if $k$ is substituted in \eqref{v eqn} by $k+\frac12$
and $v$ by $w$, the equation \eqref{v eqn} turns into \eqref{w
eqn}. Therefore,
$$
R_1(k) = P_1\left(k+\frac12\right), \;
R_2(k) = P_2\left(k+\frac12\right),
$$
and thus as $k \rightarrow \infty$,
$$
R_1(k) = a_0 + \frac{a_1}k + O\left(\frac1{k^2}\right),
$$
$$
R_2(k) = b_0 + \frac{b_1}k + O\left(\frac1{k^2}\right),
$$
with $a_0,a_1,b_0,b_1$ defined by \eqref{cfs}.

Applying now the Birkhoff-Adams Theorem we find the asymptotics of
solutions of recurrence relations \eqref{v eqn} and \eqref{w eqn}.
This leads to the following result.

\begin{lem}             \label{asympt lemma 1}
Recurrence relations \eqref{v eqn} and \eqref{w eqn} have
solutions $v_n^+$, $v_n^-$ and $w_n^+$, $w_n^-$, respectively,
with the
following asymptotics as $k\rightarrow\infty$:\\
1.
\begin{equation}
\label{no1} v_k^{\pm},\;w_k^{\pm}=\alpha_{\pm}^k
k^{\beta_{\pm}}\left(1+O\left(\frac1k\right)\right),
\end{equation}
if $\left|\frac{c_1^2+c_2^2-1}{c_1c_2}\right|\neq2$, where
$\alpha_+$ and $\alpha_-$ are the roots of the equation
$\alpha^2+a_0 \alpha+ b_0 = 0$ and
$\beta_{\pm}=\frac{a_1\alpha_{\pm}+b_1}{a_0 \alpha_{\pm}+2b_0}$
with $a_0,a_1,b_0,b_1$ defined by \eqref{cfs}.
\\Moreover, if $\left| \frac{c_1^2+c_2^2-1}{c_1c_2} \right|> 2$ then
$\alpha_{\pm}$ are real and $|\alpha_-|<1<|\alpha_+|$,
\\whereas if $\left| \frac{c_1^2+c_2^2-1}{c_1c_2} \right| < 2$ then
$\alpha_+=\overline{\alpha_-}$, $\beta_+=\overline{\beta_-}$ and
the vectors $v^+,\; v^-,\; w^+,\; w^-$ are not in $l^2(\mathbb{N})$.\\
2.
\begin{equation}
\label{no2} v_k^{\pm},\;w_k^{\pm}=\alpha^k k^{-\frac14}
e^{\delta_{\pm}\sqrt{k}}\left(1+O\left(\frac1{\sqrt{k}}\right)\right),
\end{equation}
if $\left|\frac{c_1^2+c_2^2-1}{c_1c_2}\right|=2$ and
$\lambda\neq\frac12$, where $\alpha=\alpha_+=\alpha_-$.
\\Moreover, if $ \frac{c_1^2+c_2^2-1}{c_1c_2}=2$, then
$\delta_+=2\sqrt{\frac{2\lambda-1}{2c_1c_2}}=-\delta_-$,
\\whereas if $ \frac{c_1^2+c_2^2-1}{c_1c_2}=-2$, then
$\delta_+=2\sqrt{\frac{1-2\lambda}{2c_1c_2}}=-\delta_-$.
\end{lem}

\begin{proof}
Consider recurrence relation \eqref{v eqn} and let the constants
$a_0, \  a_1, \  b_0, \  b_1$ be defined by \eqref{cfs}. Consider
the characteristic equation $\alpha^2+a_0\alpha+b_0=0$. It has
different roots, $\alpha_-<\alpha_+$, when the discriminant $D$
differs from zero: $D=\left(\frac{c_1^2+c_2^2-1}{c_1c_2}\right)^2-4 \neq 0$.
Note that $\alpha_+\alpha_-=1$.

Consider the case $D<0$. A direct application of the
Birkhoff-Adams Theorem yields:
$$
v_k^\pm=\alpha_{\pm}^k k^{\beta_{\pm}}
\left(1+O\left(\frac1k\right)\right), \; k \rightarrow \infty,
$$
where
$\beta_{\pm}=\frac{a_1\alpha_{\pm}+b_1}{a_0\alpha{\pm}+2b_0}$.
Then
$\alpha_+=\overline{\alpha_-}$, $|\alpha_+|=|\alpha_-|=1$ and
$\beta_+=\overline{\beta_-}$. Note also, that $v^{\pm}$ are not in
$l^2$:
$$
Re\ \beta_+=Re\ \beta_-=-\frac12+\frac{2\lambda-1}{2c_1c_2}Re\ \left(\frac1{a_0+2\alpha_-}\right)=-\frac12.
$$

In the case $D>0$, $\alpha_+$ and $\alpha_-$ are real and
$|\alpha_-|<1<|\alpha_+|$, hence $v^-$ lies in $l^2$.

Ultimately, in the case $D=0$, the roots of the characteristic equation coincide and
are equal to $\alpha=-\frac{a_0}2$, with $|\alpha|=1$, and the additional
condition $a_0a_1\neq2b_1$ is equivalent to
$$
-\frac{a_0^2}2+\frac{a_0(\lambda-\frac12)}{c_1c_2}\neq-2
\Leftrightarrow\ \lambda\neq\frac12.
$$
The Birkhoff-Adams Theorem yields:
$$
v_k^{\pm}=\alpha^k k^{\beta}e^{\delta_{\pm}\sqrt
k}\left(1+O\left(\frac1{\sqrt k}\right)\right), \; k \rightarrow
\infty,
$$
where $\beta=-\frac14, \
\delta_+=2\sqrt{\frac{a_0(\lambda-\frac12)}{2c_1c_2}}=-\delta_-$.
If the value $\delta_+$ is pure imaginary, then clearly the
vectors $v^{\pm}$ do not belong to $l^2(\mathbb{N})$.

In order to prove the assertion of the Lemma in relation to
$w^{\pm}$, note that in our calculations we use only the first two
orders of the asymptotical expansions for $P_1(k)$ and $P_2(k)$.
These coincide with the ones for $R_1(k)$ and $R_2(k)$. Thus, the
solutions of recurrence relations \eqref{v eqn} and \eqref{w eqn}
coincide in their main orders, which completes the proof.
\end{proof}

Now we are able to solve the recurrence relation \eqref{u eqn}
combining the solutions of recurrence relations \eqref{v eqn} and
\eqref{w eqn}.

\begin{lem}                                                         \label{asympt lemma 2}
Recurrence relation \eqref{u eqn} has two linearly independent
solutions $u_n^{+}$ and $u_n^{-}$ with the following asymptotics
as $k \rightarrow \infty$:\\
1.
$$ \left \{
\begin{array}{l}
u_{2k-1}^{\pm}=\alpha_{\pm}^k
k^{\beta_{\pm}}\left(1+O\left(\frac1k\right)\right),
\\
u_{2k}^{\pm}=-(c_1+\alpha_{\pm}c_2)\alpha_{\pm}^k
k^{\beta_{\pm}}\left(1+O\left(\frac1k\right)\right),
\end{array}
\right.
$$
if $\left| \frac{c_1^2+c_2^2-1}{c_1c_2} \right|\neq 2$, where the
values $\alpha_+,\; \alpha_-,\; \beta_+,\; \beta_-$ are taken from
the statement of Lemma \ref{asympt lemma 1}.\\
2.
$$ \left \{
\begin{array}{l}
u_{2k-1}^{\pm}=\alpha^k
e^{\delta_{\pm}\sqrt{k}}k^{-\frac14}\left(1+O\left(\frac1{\sqrt
k}\right)\right),
\\
u_{2k}^{\pm}=-(c_1+\alpha c_2)\alpha^k
e^{\delta_{\pm}\sqrt{k}}k^{-\frac14}\left(1+O\left(\frac1{\sqrt
k}\right)\right),
\end{array}
\right.
$$
if $\left| \frac{c_1^2+c_2^2-1}{c_1c_2} \right|=2$ and $\lambda
\neq \frac12$, where the values $\alpha,\; \delta_+,\; \delta_-$
are taken from the statement of Lemma \ref{asympt lemma 1}.
\end{lem}

\begin{proof}
It is clear, that any solution of recurrence relation \eqref{u
eqn} $u$ gives two vectors, $v$ and $w$, constructed of its odd
and even components, which solve recurrence relations \eqref{v
eqn} and \eqref{w eqn}, respectively. Consequently, any solution
of the recurrence relation \eqref{u eqn} belongs to the linear
space with the basis $\{V^+, \ V^-, \ W^+, \ W^-\}$, where
$$
V_{2k-1}^{\pm}=v_k^{\pm}, \ V_{2k}^{\pm}=0
\text{ and }
W_{2k-1}^{\pm}=0, \ W_{2k}^{\pm}=w_k^{\pm}.
$$
This 4-dimensional linear space contains 2-dimensional subspace of
solutions of recurrence relation \eqref{u eqn}. In order to obtain
a solution $u$ of \eqref{u eqn}, one has to obtain two conditions
on the coefficients $a_+, \ a_-, \ b_+, \ b_-$ such that $
u=a_+V^+ + a_-V^- + b_+W^+ + b_-W^-$,
\begin{equation}                                            \label{dec}
u_{2k-1}=a_+v_k^++a_-v_k^-,\; u_{2k}=b_+w_k^++b_-w_k^-.
\end{equation}
Using Lemma \ref{asympt lemma 1}, we substitute the asymptotics of
this $u$ into \eqref{u eqn} where $n$ is taken equal to $2k$,
\begin{equation*}
\lambda_{2k-1}u_{2k-1}+(q_{2k}-\lambda)u_{2k}+\lambda_{2k}u_{2k+1}=0.
\end{equation*}
As in Lemma \ref{asympt lemma 1}, we have two distinct cases.

Consider the case
$\left|\frac{c_1^2+c_2^2-1}{c_1c_2}\right|\neq2$. Then
\begin{multline*}
\left(c_1\left[a_+\left(\frac{\alpha_+}{\alpha_-}\right)^kk^{(\beta_+-\beta-)}+a_-\right]
+\left[b_+\left(\frac{\alpha_+}{\alpha_-}\right)^kk^{(\beta_+-\beta-)}+b_-\right]+\right.
\\
+\left.c_2\left[a_+\left(\frac{\alpha_+}{\alpha_-}\right)^k\alpha_
+k^{(\beta_+-\beta-)}+a_-\alpha_-\right]\right)(1+o(1))=0 \text{
as } k\rightarrow \infty.
\end{multline*}
Therefore, for any number $k$ greater than some big enough
positive $K$ one has:
$$
[c_1a_++b_++c_2a_+\alpha_+]\left(\frac{\alpha_+}{\alpha_-}\right)^kk^{(\beta_+-\beta_-)}
+[c_1a_-+b_-+c_2a_-\alpha_-]=0,
$$
hence $b_{\pm}=-(c_1+\alpha_{\pm}c_2)a_{\pm}$. Thus \eqref{dec}
admits the following form:
$$
u_{2k-1}=a_+v_k^++a_-v_k^-,
$$
$$
u_{2k}=-(c_1+\alpha_+c_2)a_+w_k^+-(c_1+\alpha_-c_2)a_-w_k^-.
$$
It is clear now, that the vectors $u^+$ and $u^-$ defined as
follows:
$$
u_{2k-1}^+=v_k^+,\; u_{2k}^+=-(c_1+\alpha_+c_2)w_k^+,
$$
$$
u_{2k-1}^-=v_k^-,\; u_{2k}^-=-(c_1+\alpha_-c_2)w_k^-,
$$
are two linearly independent solutions of the recurrence relation
\eqref{u eqn}.

The second case here,
$\left|\frac{c_1^2+c_2^2-1}{c_1c_2}\right|=2$, can be treated in
an absolutely analogous fashion.
\end{proof}

Due to Gilbert-Pearson-Khan subordinacy theory \cite{GP},
\cite{KP}, we are now ready to prove our main result concerning
the spectral structure of the operator $J$.

\begin{thm}                                                         \label{main theorem}
Depending on the modulation parameters $c_1$ and $c_2$, there are
four distinct cases, describing the spectral structure of the
operator $J$:\\
(a) If $\left|\frac{c_1^2+c_2^2-1}{c_1c_2}\right|<2$, the spectrum
is purely absolutely continuous
with local multiplicity one almost everywhere on $\mathbb{R}$,\\
(b) If $||c_1|-|c_2||=1$ and $c_1c_2 \neq 0$, the spectrum is purely
absolutely continuous with local multiplicity one almost
everywhere on $(-\infty;\frac12)$ and pure
point on $(\frac12;+\infty)$,\\
(c) If $|c_1|+|c_2|=1$ and $c_1c_2 \neq 0$, the spectrum is purely
absolutely continuous with local multiplicity one almost
everywhere on $(\frac12;+\infty)$ and pure
point on $(-\infty;\frac12)$,\\
(d) If $\left|\frac{c_1^2+c_2^2-1}{c_1c_2}\right|>2$, the spectrum
is pure point.
\end{thm}

The four cases described above are illustrated by fig. 1.
\unitlength=1mm

\begin{figure}[h] \label{fig}
\begin{picture}(100,100)
\put(50,0){\vector(0,1){100}} \put(0,50){\vector(1,0){100}}
\put(25,0){\line(1,1){75}} \put(0,25){\line(1,1){75}}
\put(0,75){\line(1,-1){75}} \put(25,100){\line(1,-1){75}}
\put(50,50){\llap{0}} \put(50,75){\llap{1}}
\put(50,100){\llap{$c_2$}} \put(75,50){1} \put(100,50){$c_1$}
\put(75,25){a} \put(75,75){a} \put(25,25){a} \put(25,75){a}
\put(12,37){b} \put(12,62){b} \put(37,12){b} \put(37,87){b}
\put(62,12){b} \put(62,87){b} \put(87,37){b} \put(87,62){b}
\put(37,37){c} \put(37,62){c} \put(62,37){c} \put(62,62){c}
\put(12,50){d} \put(50,12){d} \put(55,37){d} \put(55,87){d}
\put(87,50){d}
\end{picture}
\end{figure}

\begin{proof}
Without loss of generality, assume that $c_1$, $c_2>0$. Changing
the sign of $c_1$ or $c_2$ leads to an unitarily equivalent
operator.

Consider subordinacy properties of generalized eigenvectors
\cite{KP}.

If $\frac{|c_1^2+c_2^2-1|}{c_1c_2}>2$, we have
$|\alpha_-|<1<|\alpha_+|$. By Lemma \ref{asympt lemma 2}, $u_-$ is
a subordinate solution and lies in $l^2(\mathbb{N})$. Thus, every
real $\lambda$ can either be an eigenvalue or belong to the
resolvent set of the operator $J$.

If $\frac{|c_1^2+c_2^2-1|}{c_1c_2}<2$, we have $Re\ \alpha_+=Re\
\alpha_-,\ Re\ \beta_+=Re\ \beta_-,\ |u_n^+|\sim|u_n^-|$ as
$n\rightarrow\infty$, and there is no subordinate solution for all
real $\lambda$. The spectrum of $J$ in this situation is purely
absolutely continuous.

If $\frac{c_1^2+c_2^2-1}{c_1c_2}=2$, which is equivalent to
$|c_1-c_2|=1$, then either $\lambda>\frac12$ or $\lambda<\frac12$.
    If $\lambda > \frac12$, then $|\alpha|=1,\; \delta_+=-\delta_->0$ and
     $u_-$ is subordinate and lies in $l^2(\mathbb{N})$, hence $\lambda$ can
     either be an eigenvalue or belong to the resolvent set.
    If $\lambda < \frac12$, then $|\alpha|=1$, both $\delta_+$ and $\delta_-$
     are pure imaginary, $|u_n^+|\sim|u_n^-|$ as $n\rightarrow\infty$,
     no subordinate solution exists and ultimately $\lambda$ belongs to
     purely absolutely continuous spectrum.

If $\frac{c_1^2+c_2^2-1}{c_1c_2}=-2$, which is equivalent to $c_1+c_2=1$, the subcases
$\lambda>\frac12$ and $\lambda<\frac12$ change places, which
completes the proof.
\end{proof}

These results elaborate the domain structure, described in Section
1: we have obtained the information on the spectral structure of
the operator $J$ when the modulation parameters are on the
boundaries of regions.

\section{Criterion of semiboundedness and discreteness of the spectrum}
We start with the following Theorem which constitutes a criterion
of semiboundedness of the operator $J$.

\begin{thm}
Let $c_1c_2\neq0$.\\
1. If $|c_1|+|c_2|>1$, then the operator $J$ is not semibounded.\\
2. If $|c_1|+|c_2|\leq1$, then the operator $J$ is semibounded from below.
\end{thm}

\begin{proof}
Due to Theorem \ref{main theorem}, there are four distinct cases
of the spectral structure of the operator $J$, depending on the
values of parameters $c_1$ and $c_2$ (see fig. 1).

The case (a), i.e., $\left|\frac{c_1^2+c_2^2-1}{c_1c_2}\right|<2$,
is trivial, since $\sigma_{ac}(J)=\mathbb{R}$.

We are going to prove the assertion in the case (d), i.e.,
$\left|\frac{c_1^2+c_2^2-1}{c_1c_2}\right|>2$, using the result
 of Janas and Naboko \cite{JNr}. According to them, semiboundedness of the operator
$J$ depends on the location of the point zero relative to the
spectrum of the periodic operator $J_{per}$ (\cite{JNr}, see also
Section 1).

It is easy to see, that the absolutely continuous spectrum of the
operator $J_{per}$ in our case consists of two intervals,
\begin{gather*}
\sigma_{ac}(J_{per})=[\lambda_{-+};\lambda_{--}]\bigcup[\lambda_{+-};\lambda_{++}],\\
\text{where }\lambda_{\pm+}=1\pm(|c_1|+|c_2|), \ \lambda_{\pm-}=1\pm||c_1|-|c_2|| \text{ and}\\
\lambda_{-+}<\lambda_{--}<1<\lambda_{+-}<\lambda_{++}.
\end{gather*}
As it was established in \cite{JNr}, if the point zero lies in the
gap between the intervals of the absolutely continuous spectrum of
the operator $J_{per}$, then the operator $J$ is not semibounded.
If, on the other hand, the point zero lies to the left of the
spectrum of the operator $J_{per}$, then the operator $J$ is
semibounded from below. A direct application of this result
completes the proof in the case (d).

We now pass over to the cases (b) and (c), i.e.,
$\left|\frac{c_1^2+c_2^2-1}{c_1c_2}\right|=2, \ c_1c_2\neq0$. This
situation is considerably more complicated, since the point zero
lies right on the edge of the absolutely continuous spectrum of
the operator $J_{per}$. We consider the cases (b) and (c)
separately.

(b):

We have to prove, that the operator $J$ is not semibounded. By
Theorem \ref{main theorem}, $\sigma_{ac}(J)=(-\infty;\frac12]$,
thus the operator $J$ is not semibounded from below. Now consider
the quadratic form of the operator, taken on the canonical basis
element $e_n$. We have
$$
(Je_n,e_n)=q_n\rightarrow+\infty,\; n\rightarrow\infty,
$$
thus the operator $J$ is not semibounded.

(c):

We will show that the operator $J$ is semibounded from below. To
this end, we estimate its quadratic form: for any $u \in D(J)$
($D(J)$ being the domain of the operator $J$) one has
\begin{multline*}
(Ju,u)=\sum_{n=1}^{\infty}n|u_n|^2+\sum_{k=1}^{\infty}c_1(2k-1)(u_{2k-1}\overline{u_{2k}}+\\
+\overline{u_{2k-1}}u_{2k})+\sum_{k=1}^{\infty}c_2(2k)(u_{2k}\overline{u_{2k+1}}+\overline{u_{2k}}u_{2k+1}).
\end{multline*}
Using the Cauchy inequality \cite{Gl} and taking into account,
that $|c_1|+|c_2|=1$, we ultimately arrive at the estimate
\begin{multline}                                                            \label{estimate}
(Ju,u)\geq\sum_{n=1}^{\infty}n|u_n|^2
-\sum_{k=1}^{\infty}(|c_1|(2k-1)|u_{2k-1}|^2+|c_1|(2k-1)|u_{2k}|^2)-\\
-\sum_{k=1}^{\infty}(|c_2|(2k)|u_{2k}|^2+|c_2|(2k)|u_{2k+1}|^2)=\\
=\sum_{k=1}^{\infty}(|c_1||u_{2k}|^2+|c_2||u_{2k-1}|^2)\geq
\min\{|c_1|,|c_2|\}\|u\|^2>0,
\end{multline}
which completes the proof.
\end{proof}

The remainder of the present Section is devoted to the proof of
discreteness of the operator's pure point spectrum in the case (c)
of Theorem \ref{main theorem}, i.e., when $|c_1|+|c_2|=1, \
c_1c_2\neq0$.

By Theorem \ref{main theorem}, in this situation the absolutely
continuous spectrum covers the interval $[\frac12;+\infty)$ and
the remaining part of the spectrum, if it is present, is of pure
point type. The estimate \eqref{estimate} obtained in the proof of
the previous Theorem implies that there is no spectrum in the
interval $(-\infty; \min \{ |c_1|,|c_2| \}$). We will prove that
nonetheless if $|c_1|\neq|c_2|$, the pure point spectral component
of the operator $J$ is non-empty.

It is clear, that if $|c_1|=|c_2|=\frac12$, the spectrum of the
operator $J$ in the interval $(-\infty;\frac12)$ is empty and the
spectrum in the interval $(\frac12;+\infty)$ is purely absolutely
continuous. This situation together with its generalization towards Jacobi matrices with
zero row sums was considered by
Dombrowski and Pedersen in \cite{DP} and absolute continuity of the spectrum was established.

\begin{thm}                                                             \label{nonempty theorem}
In the case (c), i.e., when $|c_1|+|c_2|=1, \ c_1c_2\neq0$, under
an additional assumption $|c_1| \neq |c_2|$ the spectrum of the
operator $J$ in the interval $(-\infty;\frac12)$ is non-empty.
\end{thm}

\begin{proof}
Without loss of generality, assume that $0<c_1, \ c_2<1$. Changing
the sign of $c_1$, $c_2$ or both leads to an unitarily equivalent
operator.

Consider the quadratic form of the operator $J-\frac12I$ for $u \in D(J)$.
\begin{multline*}
\left(\left(J-\frac12I\right)u,u\right)
=\sum^{\infty}_{n=1}\left[q_n|u_n|^2+\lambda_n(u_n\overline{u_{n+1}}+\overline{u_n}u_{n+1})-\frac12|u_n|^2\right]=  \\
=\sum^{\infty}_{n=1}\left[n|u_n|^2+c_n n(|u_{n+1}+u_n|^2-|u_{n+1}|^2-|u_n|^2)-\frac12|u_n|^2\right].
\end{multline*}
Shifting the index $n$ by $1$ in the term $c_n n |u_{n+1}|^2$ and
then using the 2-periodicity of the sequence $\{c_n\}$, we have
\begin{multline*}
\left(\left(J-\frac12I\right)u,u\right)=\\
=\sum^{\infty}_{n=1}[c_nn|u_{n+1}+u_n|^2]+
\sum^{\infty}_{n=1}\left[n|u_n|^2-c_n n|u_n|^2-c_{n+1}(n-1)|u_n|^2-\frac12|u_n|^2\right]=  \\
=\sum^{\infty}_{n=1}[c_n n|u_{n+1}+u_n|^2]+\sum_{n=1}^{\infty}\left(\frac{c_{n+1}-c_n}2\right)|u_n|^2=  \\
=\sum_{k=1}^{\infty}[c_1(2k-1)|u_{2k-1}+u_{2k}|^2+c_2(2k)|u_{2k}+u_{2k+1}|^2]-\\
-\left(\frac{c_1-c_2}2\right)\sum_{k=1}^{\infty}[|u_{2k-1}|^2-|u_{2k}|^2].
\end{multline*}
We need to find a vector $u \in D(J)$ which makes this expression
negative. The following Lemma gives a positive answer to this
problem via an explicit construction and thus completes the proof.

\begin{lem}                                             \label{ineq lemma}
For $0<c_1, \ c_2 <1,\  c_1+c_2=1$ there exists a vector $u \in
l_{fin}(\mathbb{N})$ such that
\begin{multline}            \label{inequality}
\sum_{k=1}^{\infty}[c_1(2k-1)|u_{2k-1}+u_{2k}|^2+c_2(2k)|u_{2k}+u_{2k+1}|^2]<\\
<\left(\frac{c_1-c_2}2\right)\sum_{k=1}^{\infty}[|u_{2k-1}|^2-|u_{2k}|^2]
\end{multline}
\end{lem}

\begin{proof}
We consider the cases $c_1>c_2$ and $c_1<c_2$ separately.
Below we will see, that the latter can be reduced to the former.

1. $c_1-c_2>0$.

In this case, we will choose a vector $v$ from
$l_{fin}(\mathbb{N})$ with nonnegative components such that if the
vector $u$ is defined by $u_{2k-1}=v_k, \ u_{2k}=-v_{k+1}$, the
condition \eqref{inequality} holds true. In terms of such $v$, the
named condition admits the following form:
\begin{equation}            \label{ineq 1}
c_1\sum_{k=1}^{\infty}[(2k-1)(v_k-v_{k+1})^2]<\left(\frac{c_1-c_2}2\right)v_1^2
\end{equation}

2. $c_1-c_2<0$.

In this case, we will choose a vector $w \in l_{fin}(\mathbb{N})$
with nonnegative components and the value $t$ such that if the
vector $u$ is defined by $u_{2k}=-w_k, \ u_{2k-1}=w_k, \
u_1=tw_1$, the condition \eqref{inequality} holds true. In terms
of $w$ and $t$, condition \eqref{inequality} admits the form
\begin{equation}            \label{ineq 2}
c_2\sum_{k=1}^{\infty}[(2k)(w_k-w_{k+1})^2]<\left(-\frac{t^2}2+2c_1t+\frac{1-4c_1}2\right)w_1^2
\end{equation}
Take $t$ such that the expression in the brackets on the
right-hand side of the latter inequality is positive. This choice
is possible, if we take take the maximum of the parabola
$y(t)=-\frac{t^2}2+2c_1t+\frac{1-4c_1}2$, located at the point
$t_0=2c_1$. Then the inequality \eqref{ineq 2} admits the form
\begin{equation}            \label{ineq 3}
c_2\sum_{k=1}^{\infty}[(2k)(w_k-w_{k+1})^2]<\frac{(c_1-c_2)^2}2w_1^2
\end{equation}

We will now explicitly construct a vector $v^{(N)} \in
l_{fin}(\mathbb{N})$ such that it satisfies both \eqref{ineq 1}
and \eqref{ineq 3} for sufficiently large numbers of $N$. Consider
the sequence $v^{(N)}_{n}=\sum_{k=n}^{N}\frac1k$ for $n \leq N$ and put
$v^{(N)}_n=0$ for $n>N$. It is clear, that as $N \rightarrow +\infty$,
$$
(v^{(N)}_1)^2=\left(\sum_{k=1}^N \frac1k\right)^2\sim (\ln{N})^2,
$$
and
$$ \sum_{k=1}^{\infty}[(2k-1)(v^{(N)}_k-v^{(N)}_{k+1})^2] \sim
\sum_{k=1}^{\infty}[(2k)(v^{(N)}_k-v^{(N)}_{k+1})^2] \sim 2
\ln{N}=o\left((v^{(N)}_1)^2\right),
$$
which completes the proof of Lemma \ref{ineq lemma}.
\end{proof}

\end{proof}

We are now able to prove the discreteness of the pure point
spectral component of the operator $J$ in the case (c) of Theorem
\ref{main theorem}, which is non-empty due to Theorem \ref{nonempty theorem}.

\begin{thm}
In the case (c), i.e., when $|c_1|+|c_2|=1$ and $c_1c_2\neq0$,
under an additional assumption $|c_1| \neq |c_2|$ the spectrum of
the operator $J$ in the interval \\ $(-\infty;\min \{ |c_1|,|c_2|
\})$ is empty, the spectrum in the interval $[\min \{ |c_1|,|c_2|
\};\frac12)$ is discrete, and the following estimate holds for the
number of eigenvalues $\lambda_n$ in the interval
$(-\infty;\frac12-\varepsilon),\; \varepsilon>0$:
$$
\#\{\lambda_n:\;\lambda_n<\frac12-\varepsilon\}\leq\frac1{\varepsilon}.
$$
\end{thm}

\begin{proof}
According to the Glazman Lemma \cite{Gl}, dimension of the
spectral subspace, corresponding to the interval
$(-\infty;\frac12-\varepsilon)$, is less or equal to the
co-dimension of any subspace $H_{\varepsilon}\subset
l^2(\mathbb{N})$ such that
\begin{equation}                                            \label{cnd}
(Ju,u)\geq\left(\frac12-\varepsilon\right)\|u\|^2
\end{equation}
for any $u \in D(J)\bigcap H_{\varepsilon}$. Consider subspaces
$l^2_N$ of vectors with zero first $N$ components, i.e.,
$l^2_N:=\{u\in l^2(\mathbb{N}):\;u_1=u_2=\cdots=u_N=0\}$. For any
$\varepsilon$, $0<\varepsilon<\frac12$ we will find a number
$N(\varepsilon)$ such that for any vector from
$H_{\varepsilon}=l^2_{N(\varepsilon)}$ the inequality \eqref{cnd}
is satisfied. We consider $\varepsilon$ such that
$0<\varepsilon<\frac12$ only, since the spectrum is empty in the
interval $(-\infty;0]$ (see the estimate \eqref{estimate}). The
co-dimension of the subspace $l^2_{N(\varepsilon)}$ is
$N(\varepsilon)$, so this value estimates from above the number of
eigenvalues in the interval $(-\infty;\frac12-\varepsilon)$.

Consider the quadratic form of the operator $J$ for $u \in D(J)$:
\begin{gather*}
(Ju,u)=\sum_{n=1}^{\infty}n[|u_n|^2+c_n2Re\ (u_n\overline{u_{n+1}})]\geq\\
\geq\sum_{n=1}^{\infty}n\left[|u_n|^2-|c_n|\left( I_n|u_n|^2+\frac1{I_n}|u_{n+1}|^2\right)\right]=\\
=|u_1|^2(1-|c_1| I_1)+\sum_{n=2}^{\infty}n|u_n|^2\left[1-|c_n|I_n -|c_{n-1}|\frac1{ I_{n-1}}\frac{n-1}n\right],
\end{gather*}
where we have used the Cauchy inequality \cite{Gl} $2Re\
(u_n\overline{u_{n+1}})\leq I_n|u_n|^2+\frac1{I_n}|u_{n+1}|^2$
with the sequence $I_n>0,\;n\in\mathbb{N}$ which we will fix as
$I_n:=1-\frac{\phi}n$ in order that the expression
\begin{equation}                                                                    \label{expr}
1-|c_n| I_n-|c_{n-1}|\frac1{ I_{n-1}}\left(1-\frac1n\right)
\end{equation}
takes its simplest form. This choice cancels out the first order
with respect to $n$. The value of $\phi$ in the interval $0<
\phi<1$ will be fixed later on. We have:
\begin{equation}            \label{I(n-1)^(-1)}
\frac1{ I_{n-1}}=1+\frac{ \phi}n+ \phi_n,
\end{equation}
where $\phi_n=O\left(\frac1{n^2}\right),\;n\rightarrow\infty$.
Moreover, as can be easily seen,
$$
\phi_n=\frac{ \phi( \phi+1)}{n(n-1-\phi)}.
$$
After substituting the value of $\phi_n$ into \eqref{I(n-1)^(-1)}
and then into \eqref{expr} we obtain:
\begin{equation}                                               \label{symm expr}
1-|c_n| I_n-|c_{n-1}|\frac1{ I_{n-1}}\left(1-\frac1n\right)
=\frac1n( \phi|c_n|+(1- \phi)|c_{n-1}|)+ \theta_n
\end{equation}
with $\theta_n:=|c_{n-1}|\left(\frac{ \phi}{n^2}-
\phi_n\left(1-\frac1n\right)\right)=O\left(\frac1{n^2}\right)$ as
$n \rightarrow\infty$.

Choose $\phi$ in order to make the right hand side of expression
\eqref{symm expr} symmetric with respect to the modulation parameters
$c_1$ and $c_2$: $ \phi=\frac12$. Then
$$
1-|c_n| I_n-|c_{n-1}|\frac1{
I_{n-1}}\left(1-\frac1n\right)=\frac1{2n}+ \theta_n.
$$
Consequently,
$$
(Ju,u)\geq|u_1|^2\left(1-\frac{|c_1|}2\right)+\sum_{n=2}^{\infty}\left(\frac12+ \theta_n
n\right)|u_n|^2.
$$
Since $n \theta_n\rightarrow0$ as $n\rightarrow\infty$, we can
choose $N(\varepsilon)$ such that for any $n>N(\varepsilon)$ the
condition $n \theta_n>-\varepsilon$ holds. Thus condition
\eqref{cnd} will be satisfied for all vectors from $D(J)\bigcap
l^2_{N(\varepsilon)}$, since their first components are zeros.

The discreteness of the pure point spectrum is proved. We pass on
to the proof of the estimate for $N(\varepsilon)$. We start with $
\theta_n$:
$$
| \theta_n|\leq\left|\frac{ \phi}{n^2} -\frac{
\phi(\phi+1)}{n(n-1-
\phi)}\frac{n-1}n\right|=\frac1{2n^2}\left|1-\frac32\frac{n-1}{n-\frac32}\right|.
$$
We have
$$
n>2 \Rightarrow \left\{2>\frac{n-1}{n-\frac32}>1\right\}
\Rightarrow \left\{n| \theta_n|<\frac1{4n}\right\}.
$$
Taking $N(\varepsilon)=\frac1{\varepsilon}$,
$0<\varepsilon<\frac12$, we see that for any $n>N(\varepsilon)>2$
the condition $n \theta_n>-\varepsilon$ holds. Thus, the condition
\eqref{cnd} is satisfied for all vectors from $D(J)\bigcap
l^2_{N(\varepsilon)}$, which completes the proof.
\end{proof}

\section{The degenerate case}
Now we consider the case, when one of the modulation parameters
turns to zero (we call this case degenerate). Formally speaking,
we cannot call such matrix a Jacobi one, but this limit case is of
certain interest for us, supplementing the whole picture.

\begin{thm}
If $c_1c_2=0$, $c\neq0$ (denoting
$c:=\max\{|c_1|,|c_2|\}$), then the spectrum of the operator $J$
is the closure of the set of eigenvalues $\lambda_n$:
$$
\sigma(J)=\overline{\{\lambda_n,\;n\in\mathbb{N}\}}.
$$
The set of eigenvalues is
$$
\{\lambda_n,\; n\in \mathbb{N}\}=
\left\{
\begin{array}{l}
\{\lambda_n^+,\;\lambda_n^-,\; n \in \mathbb{N} \}\text{, if }c_1\neq0,\; c_2=0\\
\{1,\; \tilde{\lambda}_n^+,\; \tilde{\lambda}_n^-,\; n \in \mathbb{N}\}\text{, if }c_1=0,\; c_2 \neq0,
\end{array}
\right.
$$
where eigenvalues $\lambda_n^{\pm},\; \tilde{\lambda}_n^{\pm}$
have the following asymptotics:
$$
\lambda_n^+,\; \tilde{\lambda}_n^+=2(1+c)n+O(1),\;
n\rightarrow\infty,
$$
$$
\lambda_n^-=2(1-c)n+\left(c-\frac12\right)-\frac1{16cn}+O\left(\frac1{n^2}\right),\;
n\rightarrow\infty
$$
$$
\tilde{\lambda}_n^-=2(1-c)n+\left(2c-\frac32\right)-\frac1{16cn}+O\left(\frac1{n^3}\right),\;
n\rightarrow\infty,
$$
\end{thm}

\begin{proof}
When one of the parameters $c_1$ or $c_2$ is zero, the infinite
matrix consists of 2x2 (or 1x1) blocks. Thus, the operator $J$ is
an orthogonal sum of finite 2x2 (or 1x1) matrices $J_n$,
$J=\bigoplus_{n=1}^{\infty}J_n$. Then, the spectrum of the operator
$J$ is the closure of the sum of spectrums of these matrices,
$\sigma(J)=\overline{\bigcup_{n=1}^{\infty}\sigma(J_n)}$. Let us
calculate $\sigma(J_n)$.

If $c_1\neq0,\; c_2=0$, then
$$
J_n=
\begin{pmatrix}
2n-1      & c_1(2n-1) \\
c_1(2n-1) & 2n
\end{pmatrix}
$$
and $\sigma(J_n)=\{\lambda_n^+,\; \lambda_n^-\}$, where
$\lambda_n^{\pm}=\frac{4n-1\pm\sqrt{4c^2(2n-1)^2+1}}2$ and it is
easy to see that
$$
\lambda_n^{\pm}=2(1\pm c)n-\left(\frac12\pm
c\right)\pm\frac1{16cn}+ O\left(\frac1{n^2}\right),\;
n\rightarrow\infty.
$$
If $c_1=0,\; c_2\neq0$, then $J_1=1$, $\sigma(J_1)=\{1\}$,
$$
J_{n}=
\begin{pmatrix}
2n-2      & c_2(2n-2) \\
c_2(2n-2) & 2n-1 \\
\end{pmatrix}
,\ n\geq 2
$$
and $\sigma(J_{n})=\{\tilde{\lambda}_n^+,\;
\tilde{\lambda}_n^-\}$, $n\geq1$, where
$\tilde{\lambda}_n^{\pm}=\frac{4n-3\pm\sqrt{4c^2(2n-2)^2+1}}2$ and
it is easy to see that
$$
\tilde{\lambda}_n^{\pm}=2(1\pm
c)n-\left(\frac32\pm2c\right)\pm\frac1{16cn}
+O\left(\frac1{n^2}\right),\; n\rightarrow\infty,
$$
which completes the proof.
\end{proof}

\begin{rem}
From the last Theorem it follows that as $n\rightarrow \infty$,
$\lambda_n^+$, $\tilde{\lambda}_n^+ \rightarrow +\infty$.
As for $\lambda_n^-$ and $\tilde{\lambda}_n^-$, their asymptotic behavior
depends on the parameter $c$:\\
If $c>1$, then $\lambda_n^-,\ \tilde{\lambda}_n^- \rightarrow -\infty.$\\
If $c=1$, then $\lambda_n^-,\ \tilde{\lambda}_n^- \rightarrow \frac12.$\\
Finally, if $0<c<1$, then $\lambda_n^-,\;\tilde{\lambda}_n^- \rightarrow +\infty.$

Hence, if $0< c \leq 1$, the operator $J$ is semibounded from
below, and if $c>1$, the operator $J$ is not semibounded. This
clearly corresponds to results, obtained in Section 3.

When we move along the side of the boundary square (see fig. 1,
case (c)) towards one of the points $\{D_j\}_{j=1}^4=\{(1;0); \
(0;1); \ (-1;0); \ (0;-1)$\}, the absolutely continuous spectrum
covers the interval $[\frac12;+\infty)$. At the same time, at each
limit point $D_j$, $j=1,\ 2,\ 3,\ 4$, the spectrum of $J$ becomes
pure point, which demonstrates yet another phenomenon of the
spectral phase transition. Moreover, note that the spectrum at
each limit point consists of two series of eigenvalues, one going
to $+\infty$, another accumulating to the point $\lambda=\frac12$,
both points prior to the spectral phase transition having been the
boundaries of the absolutely continuous spectrum.
\end{rem}

\begin{rem}
The proof of discreteness of the spectrum in the case (c) of
Theorem \ref{main theorem} essentially involves the
semiboundedness property of the operator $J$. In the case (b) one
does not have the advantage of semiboundedness and due to that
reason the proof of discreteness supposedly becomes much more
complicated.
\end{rem}

\begin{rem}
The choice $q_n=n$ was determined by the possibility to apply the
Birkhoff-Adams technique. It should be mentioned that much more
general situation $q_n=n^{\alpha}$, $0<\alpha<1$ may be considered
on the basis of the generalized discrete Levinson Theorem. Proper
approach has been developed in \cite{JNS}, see also \cite{DN}. One
can apply similar method in our situation. Another approach which
is also valid is so-called Jordan box case and is presented in
\cite{J}.
\end{rem}

\section*{Acknowledgements} The author expresses his deep gratitude to Prof. S. N. Naboko for his constant attention to this work and for many fruitful discussions of the subject and also to
Dr. A. V. Kiselev for his help in preparation of this paper.

\end{document}